\documentclass[a4paper,12pt]{amsart}
\usepackage[T1]{fontenc}
\usepackage{amsthm}
\usepackage{amssymb} 
\usepackage[margin=2cm]{geometry}
\numberwithin{equation}{section} %% Comment out for sequentially-numbered
\numberwithin{figure}{section} %% Comment out for sequentially-numbered
\theoremstyle{plain}
\theoremstyle{plain}
\newtheorem{thm}{Theorem}
  \theoremstyle{plain}
  \newtheorem{lem}[thm]{Lemma}
  \theoremstyle{plain}
  \newtheorem{cor}[thm]{Corollary}

\begin{document}

\title{An inequality for bi-orthogonal pairs}

\author{Christopher Meaney}

\address{Department of Mathematics, Faculty of Science, Macquarie University, North Ryde NSW 2109, Australia}

\email{chrism@maths.mq.edu.au}



\begin{abstract}
We use Salem's method \cite{Salem41,MR0070756} to prove that there
is a   lower bound for partial sums of series of bi-orthogonal
vectors in a Hilbert space, or the dual vectors. This is applied to
some lower bounds on $L^{1}$ norms for orthogonal expansions. There
is also an application concerning linear algebra.
\end{abstract}
\maketitle

\section{Introduction}

Suppose that $H$ is a Hilbert space, $n\in\mathbb{N}$, and that
$J=\left\{ 1,\dots,n\right\} $ or $J=\mathbb{N}$. A pair of sets
$\left\{ v_{j}\;:\; j\in J\right\} $ and $\left\{ w_{j}\;:\; j\in J\right\} $
in $H$ are said to be \emph{a bi-orthogonal pair} when\[
\left\langle v_{j},w_{k}\right\rangle _{H}=\delta_{jk},\qquad\forall j,k\in J.\]

Theorem 1 below is the main result of this paper and is based on ideas
from Salem \cite{Salem41,MR0070756}, where Bessel's inequality is
combined with a result of Menshov \cite{49.0293.01}. Following the
proof of this theorem, we will describe Salem's method of using $L^{2}$
inequalities to produce $L^{1}$ estimates on maximal functions. Such
estimates are related to stronger results of {O}levski\u\i \cite{0321.42010},
Kashin and Szarek \cite{0857.42019}, and Bochkarev \cite{boch2008}.
We conclude with an observation about the statement of Theorem 1 in
a linear algebra setting. Some of these results were discussed in
\cite{MR2328515}, where it was shown that Salem's methods emphasised
the universality of the Rademacher-Menshov Theorem.

\section{The Main Result}
\begin{thm}
\label{thm:main}There is a positive constant $c$ with the following
property. For every $n\ge1$, every Hilbert space $H$, and every
bi-orthogonal pair $\left\{ v_{1},\dots,v_{n}\right\} $ and $\left\{ w_{1},\dots,w_{n}\right\} $
in $H$,\begin{equation}
\log n\le c\,\max_{1\le m\le n}\left\Vert w_{m}\right\Vert _{H}\,\max_{1\le k\le n}\left\Vert \sum_{j=1}^{k}v_{j}\right\Vert _{H}.\label{eq:biorth}\end{equation}

\end{thm}

\subsubsection*{Proof}

Equip $[0,1]$ with Lebesgue measure $\lambda$ and let $V=L^{2}\left([0,1],H\right)$
be the space of $H$-valued square integrable functions on $[0,1]$,
with inner product\[
\left\langle F,G\right\rangle _{V}=\int_{0}^{1}\left\langle F(x),G(x)\right\rangle _{H}dx\]
and norm\[
\left\Vert F\right\Vert _{V}=\left(\int_{0}^{1}\left\Vert F(x)\right\Vert _{H}^{2}dx\right).\]
Suppose that $\left\{ F_{1},\dots,F_{n}\right\} $ is an orthonormal
set in $L^{2}\left([0,1]\right)$ and define vectors $p_{1},\dots,p_{n}$
in $V$ by\[
p_{k}(x)=F_{k}(x)w_{k},\qquad1\le k\le n,x\in[0,1].\]
Then \[
\left\langle p_{k}(x),p_{j}(x)\right\rangle _{H}=F_{k}(x)\overline{F_{j}(x)}\,\left\langle w_{k},w_{j}\right\rangle _{H},\qquad1\le j,k\le n,\]
and so $\left\{ p_{1},\dots,p_{n}\right\} $ is an orthogonal set
in $V$. For every $P\in V$, Bessel's Inequality says that \begin{equation}
\sum_{k=1}^{n}\frac{\left|\left\langle P,p_{k}\right\rangle _{V}\right|^{2}}{\left\Vert w_{k}\right\Vert _{H}^{2}}\le\left\Vert P\right\Vert _{V}^{2}.\label{eq:bess}\end{equation}
Note that here \[
\left\langle P,p_{k}\right\rangle _{V}=\int_{0}^{1}\left\langle P(x),w_{k}\right\rangle _{H}\overline{F_{k}(x)}dx,\qquad1\le k\le n.\]

Now consider a decreasing sequence $f_{1}\ge f_{2}\ge\cdots\ge f_{n}\ge f_{n+1}=0$
of characteristic functions of measurable subsets of $[0,1]$. For
each scalar-valued $G\in L^{2}([0,1])$ define an element of $V$
by setting\[
P_{G}(x)=G(x)\sum_{j=1}^{n}f_{j}(x)v_{j}.\]
 The Abel transformation shows that \[
P_{G}(x)=G(x)\sum_{k=1}^{n}\Delta f_{k}(x)\sigma_{k},\]
where $\Delta f_{k}=f_{k}-f_{k+1}$ and $\sigma_{k}=\sum_{j=1}^{k}v_{j}$,
for $1\le k\le n$. The functions $\Delta f_{1},\dots,\Delta f_{n}$
are characteristic functions of mutually disjoint subsets of $[0,1]$
and for each $0\le x\le1$ at most one of the values $\Delta f_{k}(x)$
is non-zero. Notice that \[
\left\Vert P_{G}(x)\right\Vert _{H}^{2}=\left|G(x)\right|^{2}\sum_{k=1}^{n}\Delta f_{k}(x)\left\Vert \sigma_{k}\right\Vert _{H}^{2}.\]
 Integrating over $[0,1]$ gives \[
\left\Vert P_{G}\right\Vert _{V}^{2}\le\left\Vert G\right\Vert _{2}^{2}\max_{1\le k\le n}\left\Vert \sigma_{k}\right\Vert _{H}^{2}.\]
Note that \[
\left\langle P_{G}(x),p_{k}(x)\right\rangle _{H}=G(x)f_{k}(x)\overline{F_{k}(x)}\,\left\langle v_{k},w_{k}\right\rangle _{H},\quad1\le k\le n,\]
and \[
\left\langle P_{G},p_{k}\right\rangle _{V}=\int_{0}^{1}G(x)f_{k}(x)\overline{F_{k}(x)}dx\,\left\langle v_{k},w_{k}\right\rangle _{H},\quad1\le k\le n.\]
Combining this with Bessel's Inequality \eqref{eq:bess}, we arrive
at the inequality\begin{equation}
\sum_{k=1}^{n}\left|\int_{[0,1]}Gf_{k}\overline{F_{k}}d\lambda\right|^{2}\frac{1}{\left\Vert w_{k}\right\Vert _{H}^{2}}\le\left\Vert G\right\Vert _{2}^{2}\max_{1\le k\le n}\left\Vert \sigma_{k}\right\Vert _{H}^{2}.\label{eq:bigbess}\end{equation}
This implies that\begin{equation}
\left(\sum_{k=1}^{n}\left|\int_{[0,1]}Gf_{k}\overline{F_{k}}d\lambda\right|^{2}\right)\le\left(\max_{1\le j\le n}\|w_{k}\|_{H}^{2}\right)\left\Vert G\right\Vert _{2}^{2}\left(\max_{1\le k\le n}\left\Vert \sigma_{k}\right\Vert _{H}^{2}\right).\label{eq:mainpoint}\end{equation}
We now concentrate on the case where the functions $F_{1},\dots,F_{n}$
are given by Menshov's result (Lemma 1 on page 255 of Kashin and Saakyan\cite{MR1007141}.)
There is a constant $c_{0}>0$, independent of $n$, so that \begin{equation}
\lambda\left(\left\{ x\in[0,1]\;:\;\max_{1\le j\le n}\left|\sum_{k=1}^{j}F_{k}(x)\right|>c_{0}\log(n)\,\sqrt{n}\right\} \right)\ge\frac{1}{4}.\label{eq:menmain}\end{equation}
Let us use $\mathcal{M}(x)$ to denote the maximal function \[
\mathcal{M}(x)=\max_{1\le j\le n}\left|\sum_{k=1}^{j}F_{k}(x)\right|,\qquad0\le x\le1.\]
 Define an integer-valued function $m(x)$ on $[0,1]$ by \[
m(x)=\min\left\{ m\;:\;\left|\sum_{k=1}^{m}F_{k}(x)\right|=\mathcal{M}(x)\right\} .\]
Furthermore, let $f_{k}$ be the characteristic function of the subset
\[
\left\{ x\in[0,1]\;:\; m(x)\ge k\right\} .\]
Then \[
\sum_{k=1}^{n}f_{k}(x)F_{k}(x)=S_{m(x)}(x)=\sum_{k=1}^{m(x)}F_{k}(x),\qquad\forall0\le x\le1.\]
For an arbitrary $G\in L^{2}\left([0,1]\right)$ we have \[
\int_{0}^{1}G(x)\overline{S_{m(x)}(x)}dx=\sum_{k=1}^{n}\int_{0}^{1}G(x)f_{k}(x)\overline{F_{k}(x)}dx.\]
Using the Cauchy-Schwarz inequality on the right hand side, we have
\begin{equation}
\left|\int_{0}^{1}G(x)\overline{S_{m(x)}(x)}dx\right|\le\sqrt{n}\left(\sum_{k=1}^{n}\left|\int_{0}^{1}Gf_{k}\overline{F_{k}}\, d\lambda\right|^{2}\right)^{1/2},\label{eq:gsf}\end{equation}
for all $G\in L^{2}([0,1])$. We will use the function $G$ which
has $\left|G(x)\right|=1$ everywhere on $[0,1]$ and with \[
G(x)\overline{S_{m(x)}(x)}=\mathcal{M}(x),\qquad\forall0\le x\le1.\]
 In this case, the left hand side of \eqref{eq:gsf} is \[
\left\Vert \mathcal{M}\right\Vert _{1}\ge\frac{c_{0}}{4}\log(n)\sqrt{n},\]
because of \eqref{eq:menmain}.Combining this with \eqref{eq:gsf}
we have \[
\frac{c_{0}}{4}\log(n)\sqrt{n}\le\sqrt{n}\left(\sum_{k=1}^{n}\left|\int_{0}^{1}Gf_{k}\overline{F_{k}}\, d\lambda\right|^{2}\right)^{1/2}.\]
This can be put back into \eqref{eq:mainpoint} to obtain \eqref{eq:biorth}.
Notice that $\|G\|_{2}=1$ on the right hand side of \eqref{eq:bigbess}.
$\square$

\section{Applications}

\subsection{$L^{1}$ estimates }

In this section we use $H=L^{2}(X,\mu)$, for a positive measure space
$(X,\mu)$. Suppose we are given an orthonormal sequence of functions
$\left(h_{n}\right)_{n=1}^{\infty}$ in $L^{2}(X,\mu)$, and suppose
that each of the functions $h_{n}$ is essentially bounded on $X$.
Let $\left(a_{n}\right)_{n=1}^{\infty}$ be a sequence of non-zero
complex numbers and set \[
M_{n}=\max_{1\le j\le n}\|h_{j}\|_{\infty}\;\text{and }\;\mathcal{S}_{n}^{*}(x)=\max_{1\le k\le n}\left|\sum_{j=1}^{k}a_{j}h_{j}(x)\right|,\quad\text{ for }x\in X,\, n\ge1.\]

\begin{lem}
For a set of functions $\left\{ h_{1},\dots,h_{n}\right\} \subset L^{2}(X,\mu)\cap L^{\infty}(X,\mu)$
and maximal function $\mathcal{S}_{n}^{*}(x)=\max_{1\le k\le n}\left|\sum_{j=1}^{k}a_{j}h_{j}(x)\right|$,
we have\[
\left|a_{j}h_{j}(x)\right|\le2\mathcal{S}_{n}^{*}(x),\qquad\forall x\in X,1\le j\le n,\]
and \[
\frac{\left|\sum_{j=1}^{k}a_{j}h_{j}(x)\right|}{\mathcal{S}_{n}^{*}(x)}\le1,\quad\forall1\le k\le n\text{ and }x\text{ where }\mathcal{S}_{n}^{*}(x)\neq0.\]
\end{lem}
\begin{proof}
The first inequality follows from the triangle inequality and the
fact that \[
a_{j}h_{j}(x)=\sum_{k=1}^{j}a_{k}h_{k}(x)-\sum_{k=1}^{j-1}a_{k}h_{k}(x)\]
for $2\le j\le n$. The second inequality is a consequence of the
definition of $\mathcal{S}_{n}^{*}$.
\end{proof}
Fix $n\ge1$ and let \[
v_{j}(x)=a_{j}h_{j}(x)\left(\mathcal{S}_{n}^{*}(x)\right)^{-1/2}\;\text{ and }w_{j}(x)=a_{j}^{-1}h_{j}(x)\left(\mathcal{S}_{n}^{*}(x)\right)^{1/2}\]
for all $x\in X$ where $\mathcal{S}_{n}^{*}(x)\neq0$ and $1\le j\le n$.
From their definition,\[
\left\{ v_{1},\dots,v_{n}\right\} \;\text{ and }\;\left\{ w_{1},\dots,w_{n}\right\} \]
 are a bi-orthogonal pair in $L^{2}(X,\mu)$. The conditions we have
placed on the functions $h_{j}$ give:\[
\|w_{j}\|_{2}^{2}=\left|a_{j}\right|^{-2}\int_{X}\left|h_{j}\right|^{2}\left(\mathcal{S}_{n}^{*}\right)d\mu\le\frac{M_{n}^{2}}{\min_{1\le k\le n}\left|a_{k}\right|^{2}}\|\mathcal{S}_{n}^{*}\|_{1}\]
and \[
\left\Vert \sum_{j=1}^{k}v_{j}\right\Vert _{2}^{2}=\int_{X}\frac{1}{(\mathcal{S}_{n}^{*})}\left|\sum_{j=1}^{k}a_{j}h_{j}\right|^{2}d\mu\le\left\Vert \sum_{j=1}^{k}a_{j}h_{j}\right\Vert _{1}.\]
We can put these estimates into \eqref{eq:biorth} and find that \[
\log n\le c\,\frac{M_{n}}{\min_{1\le k\le n}\left|a_{k}\right|}\|S_{n}^{*}\|_{1}^{1/2}\,\max_{1\le k\le n}\left\Vert \sum_{j=1}^{k}a_{j}h_{j}\right\Vert _{1}^{1/2}.\]
 We could also say that \[
\max_{1\le k\le n}\left\Vert \sum_{j=1}^{k}a_{j}h_{j}\right\Vert _{1}\le\|\mathcal{S}_{n}^{*}\|_{1}^{}\]
and so \[
\log(n)\le c\frac{M_{n}}{\min_{1\le k\le n}\left|a_{k}\right|}\left\Vert \mathcal{S}_{n}^{*}\right\Vert _{1}.\]

\begin{cor}
\label{cor:maxmaxmax}Suppose that $\left(h_{n}\right)_{n=1}^{\infty}$
is an orthonormal sequence in $L^{2}\left(X,\mu\right)$ consisting
of essentially bounded functions. For each sequence $\left(a_{n}\right)_{n=1}^{\infty}$
of complex numbers and each $n\ge1,$\[
\left(\min_{1\le k\le n}\left|a_{k}\right|\ \log n\right)^{2}\le c\left(\max_{1\le k\le n}\left\Vert h_{k}\right\Vert _{\infty}\right)^{2}\left\Vert \max_{1\le k\le n}\left|\sum_{j=1}^{k}a_{j}h_{j}\right|\;\right\Vert _{1}\max_{1\le k\le n}\left\Vert \sum_{j=1}^{k}a_{j}h_{j}\right\Vert _{1}\]
and\[
\min_{1\le k\le n}\left|a_{k}\right|\ \log n\le c\left(\max_{1\le k\le n}\left\Vert h_{k}\right\Vert _{\infty}\right)\left\Vert \max_{1\le k\le n}\left|\sum_{j=1}^{k}a_{j}h_{j}\right|\;\right\Vert _{1}.\]
The constant $c$ is independent of $n$, and the sequences involved
here.
\end{cor}
As observed in \cite{0857.42019}, this can also be obtained as a
consequence of \cite{0321.42010}. In addition, see \cite{MR565158}.

The following is a paraphrase of the last page of \cite{Salem41}.
For the special case of Fourier series on the unit circle, see Proposition
1.6.9 in \cite{pinsky2002}.
\begin{cor}
Suppose that $\left(h_{n}\right)_{n=1}^{\infty}$ is an orthonormal
sequence in $L^{2}\left(X,\mu\right)$ consisting of essentially bounded
functions with $\left\Vert h_{n}\right\Vert _{\infty}\le M$ for all
$n\ge1$. For each decreasing sequence $\left(a_{n}\right)_{n=1}^{\infty}$
of positive numbers and each $n\ge1,$\[
\left(a_{n}\ \log n\right)^{2}\le cM^{2}\left\Vert \max_{1\le k\le n}\left|\sum_{j=1}^{k}a_{j}h_{j}\right|\;\right\Vert _{1}\max_{1\le k\le n}\left\Vert \sum_{j=1}^{k}a_{j}h_{j}\right\Vert _{1}\]
and\[
a_{n}\ \log n\le cM\left\Vert \max_{1\le k\le n}\left|\sum_{j=1}^{k}a_{j}h_{j}\right|\;\right\Vert _{1}.\]
In particular, if $\left(a_{n}\log n\right)_{n=1}^{\infty}$ is unbounded
then \[
\left(\left\Vert \max_{1\le k\le n}\left|\sum_{j=1}^{k}a_{j}h_{j}\right|\;\right\Vert _{1}\right)_{n=1}^{\infty}\text{ is unbounded.}\]
The constant $c$ is independent of $n$, and the sequences involved
here.
\end{cor}

\subsection{Salem's Approach to the Littlewood Conjecture}

We concentrate on the case where $H=L^{2}\left(\mathbb{T}\right)$
and the orthonormal sequence is a subset of $\left\{ e^{inx}\;:\; n\in\mathbb{N}\right\} $.
Let \[
m_{1}<m_{2}<m_{3}<\cdots\]
 be an increasing sequence of natural numbers and let \[
h_{k}(x)=e^{im_{k}x}\]
for all $k\ge1$ and $x\in\mathbb{T}.$ In addition, let \[
D_{m}(x)=\sum_{k=-m}^{m}e^{ikx}\]
 be the $m^{\text{th }}$ Dirichlet kernel. For all $N\ge m\ge1$,
there is the partial sum\[
\sum_{m_{k}\le m}a_{k}h_{k}(x)=D_{m}*\left(\sum_{m_{k}\le N}a_{k}h_{k}\right)(x).\]
It is a fact that $D_{m}$ is an even function which satisfies the
inequalities:\begin{equation}
\left|D_{m}(x)\right|\le\begin{cases}
2m+1 & \text{ for all }x,\\
1/|x| & \text{for }\frac{1}{2m+1}<x<2\pi-\frac{1}{2m+1}.\end{cases}\label{eq:95}\end{equation}

\begin{lem}
If $p$ is a trigonometric polynomial of degree $N$ then the maximal
function of its Fourier partial sums \[
S^{*}p(x)=\sup_{m\ge1}\left|D_{m}*p(x)\right|\]
satisfies\[
\left\Vert S^{*}p\right\Vert _{1}\le c\log\left(2N+1\right)\left\Vert p\right\Vert _{1}\]
\end{lem}
\begin{proof}
For such a trigonometric polynomial $p$, the partial sums are all
partial sums of $p*D_{N}$, and all the Dirichlet kernels $D_{m}$
for $1\le m\le N$ are dominated by a function whose $L^{1}$ norm
is of the order of $\log(2N+1)$.
\end{proof}
We can combine this with the inequalities in Corollary \ref{cor:maxmaxmax},
since \[
\left\Vert \max_{1\le k\le n}\left|\sum_{j=1}^{k}a_{j}h_{j}\right|\;\right\Vert _{1}\le c\log\left(2m_{n}+1\right)\left\Vert \sum_{j=1}^{m}a_{j}h_{j}\right\Vert _{1}.\]
We then arrive at the main result in \cite{MR0070756}.
\begin{cor}
For an increasing sequence $\left(m_{n}\right)_{n=1}^{\infty}$ of
natural numbers and a sequence of non-zero complex numbers $\left(a_{n}\right)_{n=1}^{\infty}$
the partial sums of the trigonometric series \[
\sum_{k=1}^{\infty}a_{k}e^{im_{k}x}\]
 satisfy\[
\min_{1\le k\le n}\left|a_{k}\right|\frac{\log n}{\sqrt{\log(2m_{n}+1)}}\le c\max_{1\le k\le n}\left\Vert \sum_{j=1}^{k}a_{j}e^{im_{j}(\cdot)}\right\Vert _{1}.\]

\end{cor}
This was Salem's attempt at Littlewood's conjecture, which was subsequently
settled in \cite{MR616222} and \cite{MR621019}.

\subsection{Linearly Independent Sequences}

Notice that if $\left\{ v_{1},\dots,v_{n}\right\} $ is an arbitrary
linearly independent subset of $H$ then there is a unique subset
\[
\left\{ w_{j}^{n}\;:\;1\le j\le n\right\} \subseteq\text{span}\left(\left\{ v_{1},\dots,v_{n}\right\} \right)\]
 so that $\left\{ v_{1},\dots,v_{n}\right\} $ and $\left\{ w_{1}^{n},\dots,w_{n}^{n}\right\} $are
a bi-orthogonal pair. See Theorem 15 in Chapter 3 of \cite{HoffmanKunze}.
We can apply Theorem 1 to the pair in either order.
\begin{cor}
For each $n\ge2$ and linearly independent subset $\left\{ v_{1},\dots,v_{n}\right\} $
in an inner-product space $H$ , with dual basis \textup{$\left\{ w_{1}^{n},\dots,w_{n}^{n}\right\} $,
\[
\log n\le c\,\max_{1\le k\le n}\left\Vert w_{k}^{n}\right\Vert _{H}\max_{1\le k\le n}\left\Vert \sum_{j=1}^{k}v_{j}\right\Vert _{H}\]
and \[
\log n\le c\,\max_{1\le k\le n}\left\Vert v_{k}\right\Vert _{H}\max_{1\le k\le n}\left\Vert \sum_{j=1}^{k}w_{j}^{n}\right\Vert _{H}.\]
The constant $c>0$ is independent of $n$, $H$, and the sets of
vectors. }
\end{cor}

\subsection{Matrices}

Suppose that $A$ is an invertible $n\times n$ matrix with complex
entries and columns \[
a_{1},\dots,a_{n}\in\mathbb{C}^{n}.\]
Let $b_{1},\dots,b_{n}$ be the rows of $A^{-1}$. From their definition
\[
\sum_{j=1}^{n}b_{ij}a_{jk}=\delta_{ik}\]
 and so the two sets of vectors \[
\left\{ \overline{b_{1}^{T}},\dots,\overline{b_{n}^{T}}\right\} \text{ and }\left\{ a_{1},\dots,a_{n}\right\} \]
 are a bi-orthogonal pair in $\mathbb{C}^{n}$. Theorem \ref{thm:main}
then says that \[
\log(n)\le c\,\max_{1\le k\le n}\left\Vert b_{k}\right\Vert \,\max_{1\le k\le n}\left\Vert \sum_{j=1}^{k}a_{j}\right\Vert .\]
 The norm here is the finite dimensional $\ell^{2}$ norm.

Note that \cite{0857.42019} has logarithmic lower bounds for $\ell^{1}$-norms
of column vectors of orthogonal matrices.

\bibliographystyle{amsplain}

\begin{thebibliography}{100}
	
\providecommand{\bysame}{\leavevmode\hbox to3em{\hrulefill}\thinspace}



\bibitem{HoffmanKunze}
K.~Hoffman \and R.~A. Kunze, \emph{Linear {A}lgebra}, second ed., Prentice
  Hall, 1971.

\bibitem{boch2008}
S.~V. Bochkarev, \emph{A generalization of {K}olmogorov's theorem to
  biorthogonal systems}, Proceedings of the Steklov Institute of Mathematics
  \textbf{260} (2008), 37--49.

\bibitem{MR1007141}
B.~S. Kashin and A.~A. Saakyan, \emph{Orthogonal series}, Translations of
  Mathematical Monographs, vol.~75, American Mathematical Society, Providence,
  RI, 1989.  

\bibitem{0857.42019}
\bysame and S.~J. Szarek, \emph{{Logarithmic growth of the $L\sp 1$-norm
  of the majorant of partial sums of an orthogonal series}}, Math. Notes
  \textbf{58} (1995), no.~2, 824--832.

\bibitem{MR616222}
S.~V. Konyagin, \emph{On the {L}ittlewood problem}, Izv. Akad. Nauk SSSR Ser.
  Mat. \textbf{45} (1981), no.~2, 243--265, 463.  

\bibitem{MR565158}
S.~Kwapie{\'n} and S.~J. Szarek, \emph{An estimation of
  the {L}ebesgue functions of biorthogonal systems with an application to the
  nonexistence of some bases in {$C$} and {$L\sp{1}$}}, Studia Math.
  \textbf{66} (1979), no.~2, 185--200.  

\bibitem{MR621019}
O.~C. McGehee, L.~Pigno, and B.~Smith, \emph{Hardy's inequality and
  the {$L\sp{1}$} norm of exponential sums}, Ann. of Math. (2) \textbf{113}
  (1981), no.~3, 613--618.  

\bibitem{MR2328515}
C.~Meaney, \emph{Remarks on the {R}ademacher-{M}enshov theorem},
  CMA/AMSI Research Symposium ``Asymptotic Geometric Analysis, Harmonic
  Analysis, and Related Topics'', Proc. Centre Math. Appl. Austral. Nat. Univ.,
  vol.~42, Austral. Nat. Univ., Canberra, 2007, pp.~100--110.  

\bibitem{49.0293.01}
D.~Menchoff, \emph{{Sur les s{\'e}ries de fonctions orthogonales. (Premi{\'e}re
  Partie. La convergence.)}}, Fundamenta math. \textbf{4} (1923), 82--105.

\bibitem{0321.42010}
A.~M. {O}levski\u\i, \emph{{Fourier series with respect to general orthogonal
  systems. Translated from the Russian by B. P. Marshall and H. J.
  Christoffers.}}, {Ergebnisse der Mathematik und ihrer Grenzgebiete. Band 86.
  Berlin-Heidelberg-New York: Springer-Verlag. }, 1975.

\bibitem{pinsky2002}
M.~A. Pinsky, \emph{Introduction to {F}ourier {A}nalysis and {W}avelets},
  Brooks/Cole, 2002.

\bibitem{Salem41}
R~Salem, \emph{A new proof of a theorem of {M}enchoff}, Duke Math. J.
  \textbf{8} (1941), 269--272.

\bibitem{MR0070756}
\bysame, \emph{On a problem of {L}ittlewood}, Amer. J. Math. \textbf{77}
  (1955), 535--540. 

\end{thebibliography}

 \end{document}